\documentclass[a4paper,11pt,reqno]{amsart}

          \usepackage{amssymb}
          \usepackage{amsmath}
          \usepackage{amsthm}
          \usepackage{amsfonts}
          \usepackage[english]{babel}
          \usepackage[utf8]{inputenc}
          \usepackage{enumerate}
          \usepackage{graphicx}

\usepackage{color}

\usepackage[small,nohug]{diagrams}

 \usepackage[unicode,colorlinks,plainpages=false,hyperindex=true,bookmarksnumbered=true,bookmarksopen=false,pdfpagelabels]{hyperref}
 \hypersetup{urlcolor=cyan,linkcolor=blue,citecolor=red,colorlinks=false} 

\makeatletter
\renewcommand*{\p@section}{\S\,}
\renewcommand*{\p@subsection}{\S\,}
\renewcommand*{\p@subsubsection}{\S\,}
\makeatother

\usepackage{etoolbox}
\makeatletter
\patchcmd{\@maketitle}
  {\ifx\@empty\@dedicatory}
  {\ifx\@empty\@date \else {\vskip3ex \centering\footnotesize\@date\par\vskip1ex}\fi
   \ifx\@empty\@dedicatory}
  {}{}
\patchcmd{\@adminfootnotes}
  {\ifx\@empty\@date\else \@footnotetext{\@setdate}\fi}
  {}{}{}
\makeatother

\def\BibTeX{{\rm B\kern-.05em{\sc i\kern-.025em b}\kern-.08em
    T\kern-.1667em\lower.7ex\hbox{E}\kern-.125emX}}

\newtheorem{thm}{Theorem}[section]

\newtheorem{lem}[thm]{Lemma}
\newtheorem{prop}[thm]{Proposition}

\theoremstyle{definition}
\newtheorem{defn}{Definition}[section]
\theoremstyle{remark}
\newtheorem{rem}{Remark}[section]
\newtheorem{exmp}{Example}[section]

\numberwithin{equation}{section}

\def\restriction#1#2{\mathchoice
              {\setbox1\hbox{${\displaystyle #1}_{\scriptstyle #2}$}
              \restrictionaux{#1}{#2}}
              {\setbox1\hbox{${\textstyle #1}_{\scriptstyle #2}$}
              \restrictionaux{#1}{#2}}
              {\setbox1\hbox{${\scriptstyle #1}_{\scriptscriptstyle #2}$}
              \restrictionaux{#1}{#2}}
              {\setbox1\hbox{${\scriptscriptstyle #1}_{\scriptscriptstyle #2}$}
              \restrictionaux{#1}{#2}}}
\def\restrictionaux#1#2{{#1\,\smash{\vrule height .8\ht1 depth .85\dp1}}_{\,#2}} 

\newcommand{\N}{\ensuremath{\mathbb{N}}}

\newcommand{\R}{\ensuremath{\mathbb{R}}}
\newcommand{\Z}{\ensuremath{\mathbb{Z}}}
\newcommand{\Vect}{\operatorname{Vec}}

\newcommand{\ev}{\operatorname{ev}}

\begin{document}

\title{Introduction to graded geometry}

%
\author[M. Fairon]{Maxime Fairon}
 \address{School of Mathematics\\
         University of Leeds\\
         Leeds LS2 9JT\\ United Kingdom}
 \email{mmmfai@leeds.ac.uk}

\date{\today}

\begin{abstract}
This paper aims at setting out the basics of $\Z$-graded manifolds theory. 
We introduce $\Z$-graded manifolds from local models and give some of their properties. 
The requirement to work with a completed graded symmetric algebra to define functions is made clear. 
Moreover, we define vector fields and exhibit their graded local basis. The paper also reviews some correspondences 
between differential $\Z$-graded manifolds and algebraic structures.

\smallskip
\noindent \textbf{Keywords.} Supergeometry, graded manifold, differential graded manifold, $Q$-manifold. 

\noindent \textbf{2010 MSC.} 58A50, 51-02
\end{abstract}

\maketitle

\section{Introduction}
\label{SectIntro}

 In the $70$s, supermanifolds were introduced and studied to provide a geometric background to the developing  
theory of supersymmetry. 
In the  Berezin-Leites \cite{BerLei} and Kostant \cite{Kost77} approach, they are defined as $\Z_2$-graded 
locally ringed spaces. More precisely, this means that a supermanifold is a pair $(|N|,\mathcal{O}_N)$ such that 
$|N|$ is a topological space and 
$\mathcal{O}_N$ is a sheaf of $\Z_2$-graded algebras which satisfies, for all  sufficiently small open subsets $U$,
\begin{equation}
 \mathcal{O}_N(U) \simeq \mathcal{C}^\infty_{\R^n}(U)\otimes \bigwedge \R^m\,\,, \nonumber 
\end{equation}
where $\bigwedge \R^m$ is endowed with its canonical $\Z_2$-grading. 
Accordingly, a local coordinate system on a supermanifold splits into $n$ smooth coordinates on $U$ and $m$ elements of a basis on 
$\R^m$. Both types of coordinates are distinguished via  
the parity function $p$ with values in $\Z_2=\{0,1\}$. Namely, the first 
$n$ coordinates are said to be even ($p=0$), while the others are odd ($p=1$). 

From the late $90$s, the introduction of an integer grading was necessary in some topics related to Poisson geometry, 
Lie algebroids and Courant algebroids. These structures could carry the integer grading only, as it appeared in the work of 
Kontsevitch \cite{Konts03}, Roytenberg \cite{Roy02} and Severa \cite{SeveraHomotopy}, or could be endowed 
with both $\Z_2$- and $\Z$-gradings, as it was introduced by Voronov \cite{Voronov02}. The first approach, that we follow,   
led to the definition of a $\Z$-graded manifold given by Mehta \cite{ThesisPing}. 
Similarly to a supermanifold, a $\Z$-graded manifold is a graded locally ringed space. Its structure sheaf $\mathcal{O}_N$ is a 
sheaf of $\Z$-graded algebras which is given locally by 
\begin{equation}
 \mathcal{O}_N(U) \simeq \mathcal{C}^\infty_{\R^n}(U)\otimes 
\overline{\mathbf{S}\mathcal{W}}\,\,, \nonumber 
\end{equation}
where $\mathcal{W}=\oplus_i \mathcal{W}_i$ is a real $\Z$-graded vector space whose component of degree zero satisfies 
$\mathcal{W}_0=\{0\}$. Here, $\overline{\mathbf{S}\mathcal{W}}$ denotes the algebra of formal power series of supercommutative products of 
elements in $\mathcal{W}$.  Therefore, the local coordinate system of the $\Z$-graded manifold
can be described by $n$ smooth coordinates and a graded basis of $\mathcal{W}$. In particular, it might admit generators 
of even degree but not of smooth type. This implies that such coordinates do not square to zero and can generate polynomials of 
arbitrary orders. Their existence requires the introduction of formal power series to obtain the locality of every stalk of the sheaf. 
This condition, which has not always been accurately considered in past works, 
also occurs with $\Z^n_2$-grading for $n\geqslant 2$, as discussed in \cite{Covolo}.

\vspace{0.2cm}

Complete references on supermanifolds can be found (\emph{e.g.} \cite{CCFsupersym,DelMor99,Kost77,Varadarajansupersym}), 
but mathematicians have not yet prepared monographs on graded manifolds. 
The aim  of this paper is to show what would appear at the beginning of such a book and it aspires to be a comprehensive introduction 
to $\Z$-graded manifolds theory for both mathematicians and theoretical physicists. 

The layout of this article is inspired by \cite{CCFsupersym,Varadarajansupersym} and the reader needs a little knowledge of 
sheaf theory from  \ref{SubSectGradedRingSp} onward. We first introduce graded vector spaces, graded rings and graded algebras 
in \ref{SectBasis}. These objects allow us to define graded locally ringed spaces and their morphisms. After that, 
we introduce graded domains and we show that their stalks are local. 
In particular, we notice that this property follows from the introduction of formal power series to define sections. 
The graded domains are the local models of graded manifolds, 
which are studied in \ref{SectionGradedManifold}. Specifically, we give the elementary properties of graded manifolds and define their 
vector fields. In particular, we prove that there exists a local graded basis of the vector fields which is related to the local 
coordinate system of the graded manifold. Finally, we illustrate in \ref{SectionApplication} the theory of $\Z$-graded manifolds with 
a few theorems stating the correspondence between differential graded manifolds and algebraic structures. These important  
examples are usually referred to as $Q$-manifolds. Other applications can be found in \cite{CattaneoSchatz,Qiu2011}.

\vspace{0.2cm}

\textbf{Conventions.} In this paper, $\N$ and $\Z$ denote the set of nonnegative integers and the set of integers, respectively. We write 
 $\N^\times$ and $\Z^\times$ when we consider these sets deprived of zero. 
Notice that some authors use the expression 
\emph{graded manifolds} to talk about supermanifolds with an additional $\Z$-grading (see \cite{Voronov02}), 
but we only use this expression in the present paper to refer to $\Z$-graded manifolds.

\section{Preliminaries}
\label{SectBasis}

\subsection{Graded algebraic structures}
\label{SubSectGradedAlgebra}


\subsubsection{Graded vector space}
A \emph{$\Z$-graded vector space} is a direct sum $\mathcal{V}=\bigoplus_{i}\mathcal{V}_i$ of 
a collection of $\R$-vector spaces $(\mathcal{V}_i)_{i\in\Z}$.
If a nonzero element $v\in \mathcal{V}$ belongs to one of the $\mathcal{V}_i$, one says that it is \emph{homogeneous} of degree $i$. 
We write $|\cdot|$ for the application which assigns its degree to a homogeneous element. 
Moreover, we only consider $\Z$-graded vector spaces of \emph{finite type}, 
which means that $\mathcal{V}=\bigoplus_{i}\mathcal{V}_i$ 
is such that $\text{dim}(\mathcal{V}_i)<\infty$ for all $i\in \Z$. 
A \emph{graded basis} of $\mathcal{V}$ is a sequence $(v_\alpha)_\alpha$ of homogeneous elements of $\mathcal{V}$ such that the 
subsequence of all elements $v_\alpha$ of degree $i$ is a basis of the vector space $\mathcal{V}_i$, for all $i\in\Z$.

If $\mathcal{V}$ is a $\Z$-graded vector space, $\mathcal{V}[k]$ denotes the $\Z$-graded vector space $\mathcal{V}$ lifted by $k\in\Z$ : 
$(\mathcal{V}[k])_i=\mathcal{V}_{i-k}$ for all $i\in \Z$. 
The $\Z$-graded vector space $\mathcal{V}$ with reversed degree is denoted by $\Pi\mathcal{V}$ and 
satisfies $(\Pi\mathcal{V})_i=\mathcal{V}_{-i}$ for all $i \in \Z$.

Given two $\Z$-graded vector spaces $\mathcal{V}$ and $\mathcal{W}$, their direct sum $\mathcal{V}\oplus \mathcal{W}$ 
can be defined with the grading  $(\mathcal{V}\oplus \mathcal{W})_i=\mathcal{V}_i\oplus \mathcal{W}_i$, as well as the  
tensor product $\mathcal{V}\otimes \mathcal{W}$ with 
$(\mathcal{V}\otimes \mathcal{W})_i = \oplus_{j\in\Z} \mathcal{V}_j \otimes \mathcal{W}_{i-j}$. 
Both constructions are associative.
A \emph{morphism of $\Z$-graded vector spaces} $T: \mathcal{V} \rightarrow \mathcal{W}$ is a linear map which 
preserves the degree : $T(\mathcal{V}_i) \subseteq \mathcal{W}_i$ for all $i\in\Z$. 
We write $\text{Hom}(\mathcal{V},\mathcal{W})$ for the set of all morphisms between $\mathcal{V}$ and $\mathcal{W}$. 

The category of $\Z$-graded vector spaces is a symmetric monoidal category. It is equipped with the  
nontrivial commutativity isomorphism 
$\mathbf{c}_{\mathcal{V},\mathcal{W}}:\mathcal{V}\otimes \mathcal{W}\to\mathcal{W}\otimes\mathcal{V}$ 
which, to any homogeneous elements $v \in \mathcal{V}$ and $w \in \mathcal{W}$, 
assigns the homogeneous elements $\mathbf{c}_{\mathcal{V},\mathcal{W}}(v\otimes w) = (-1)^{|v||w|} w\otimes v$.  

Finally, duals of $\Z$-graded vector spaces can be defined. Given a $\Z$-graded vector space $\mathcal{V}$, 
\underline{Hom}$(\mathcal{V},\R)$ denotes the $\Z$-graded vector space which contains all the linear maps from $\mathcal{V}$ to 
$\R$. Its grading is defined by setting  $\left(\text{\underline{Hom}}(\mathcal{V},\R)\right)_i$ to be the set of linear maps $f$ such 
that $f(v)\in \R$ if $v\in \mathcal{V}_{-i}$. It is the dual of $\mathcal{V}$ and we write 
$\mathcal{V}^*=$\underline{Hom}$(\mathcal{V},\R)$. 
One can show that the latter satisfies $(\mathcal{V}^*)_i=(\mathcal{V}_{-i})^*$, for all $i\in\Z$.

\begin{rem}
From now on, we will refer to $\Z$-graded objects simply as graded objects. However, we will keep the complete writing in the definitions 
or when it is needed, to keep it clear that the grading is taken over $\Z$.
\end{rem}

\begin{rem}
 The different notions introduced for $\Z$-graded vector spaces
extend to $\Z$-graded $R$-modules over a ring $R$. 
In particular, taking $R=\Z$, these constructions apply to any $\Z$-graded abelian group (which is a direct sum of 
abelian groups) seen as a direct sum of $\Z$-modules.
\end{rem}

\subsubsection{Graded ring}
\label{sSubSectRings}

A \emph{$\Z$-graded ring} $\mathcal{R}$ is a $\Z$-graded abelian group $\mathcal{R}=\bigoplus_{i}\mathcal{R}_i$ with a morphism 
$\mathcal{R}\otimes \mathcal{R}\to \mathcal{R}$ called the \emph{multiplication}. By definition, it satisfies 
$\mathcal{R}_i \mathcal{R}_j \subseteq \mathcal{R}_{i+j}$. 

A graded ring $\mathcal{R}$ is \emph{unital} if it admits an element $1$ such that $1\,r=r=r\, 1$ for any $r\in \mathcal{R}$. 
In that case, the element $1$ satisfies $1\in \mathcal{R}_0$. The graded ring $\mathcal{R}$ is \emph{associative} if $(ab)c=a(bc)$ 
for every $a,b,c\in \mathcal{R}$. 
The graded ring $\mathcal{R}$ is \emph{supercommutative} when $ab=(-1)^{|a||b|}ba$ for any homogeneous elements $a,b\in \mathcal{R}$. 
This means that the multiplication is invariant under the commutativity isomorphism $\mathbf{c}$.  

We can introduce the definitions of left ideal, right ideal and two-sided ideal (which is referred to as \emph{ideal}) 
of a graded ring in the same manner as in the non-graded case. It is easy to see that, 
in a supercommutative associative unital graded ring, a left (or right) ideal is an ideal. 
A \emph{homogeneous ideal} of $\mathcal{R}$ is an ideal $I$ such that $I=\bigoplus_kI_k$, for $I_k=I\cap \mathcal{R}_k$. 
Equivalently, a homogeneous ideal $I$ is an ideal generated by a set of homogeneous elements 
$H\subseteq \cup_k \mathcal{R}_k$. In that case, we write $I=\langle H \rangle$ when we want to emphasize the generating set of $I$. 
An homogeneous ideal  $I\subsetneq \mathcal{R}$ is said to 
be \emph{maximal} if, when $I\subseteq J$ with $J$ another homogeneous ideal, then either $J=I$ or $J=\mathcal{R}$. 
A \emph{local} graded ring $\mathcal{R}$ is a graded ring which admits a unique maximal homogeneous ideal.

Define $\mathcal{J}_\mathcal{R}$ as the ideal generated by the elements of $\mathcal{R}$ with nonzero degree. We can 
consider the quotient $\mathcal{R}/\mathcal{J}_\mathcal{R}$ and the projection map 
$\pi : \mathcal{R} \to \mathcal{R}/\mathcal{J}_\mathcal{R}$. 
We say that $\mathcal{R}$ is a \emph{$\pi$-local graded ring} if it admits a unique maximal homogeneous ideal $\mathfrak{m}$ 
such that $\pi(\mathfrak{m})$ is the unique maximal ideal of $\mathcal{R}/\mathcal{J}_\mathcal{R}$. Note that a local graded ring 
is always $\pi$-local, but the converse is not true (see Remark \ref{RemLocality} below). 

\begin{rem}
  The different notions introduced for $\Z$-graded rings extend to any object with an underlying $\Z$-graded ring structure.
\end{rem}

\subsubsection{Graded algebra}

A $\Z$-\emph{graded algebra} is a $\Z$-graded vector space  $\mathcal{A}$ endowed with a morphism  
$\mathcal{A}\otimes \mathcal{A}\to \mathcal{A}$ called \emph{multiplication}. Alternatively, it is a graded ring with a structure of 
$\R$-module. 
 
Consider a $\Z$-graded vector space $\mathcal{W}$.
The \emph{free $\Z$-graded associative algebra} generated by $\mathcal{W}$ is    
$\bigotimes \mathcal{W} = \bigoplus_{k\in \N}\otimes^k \mathcal{W}$. The product on $\bigotimes \mathcal{W}$ is 
the concatenation, while its degree is induced by the degree on $\mathcal{W}$.  
Then, the \emph{symmetric $\Z$-graded associative algebra} generated by $\mathcal{W}$ is given by     
$\mathbf{S} \mathcal{W}=\bigotimes \mathcal{W} / J$, where  
$J= \left\langle\{ u\otimes v -(-1)^{|u||v|} v \otimes u  \, \, |\,\, u,v\in \mathcal{W} \, \text{homogeneous} \}\right\rangle$. 

This object gathers the non-graded exterior and symmetric algebras.  
Indeed, split $\mathcal{W}$ as $\mathcal{W}=\mathcal{W}_{\text{even}}\oplus\mathcal{W}_{\text{odd}}$ with 
$\mathcal{W}_{\text{even}}=\oplus_i \mathcal{W}_{2i}$ and $\mathcal{W}_{\text{odd}}=\oplus_i \mathcal{W}_{2i+1}$. Then  
we obtain that $\mathbf{S}\mathcal{W}\simeq  \mathbf{S}\mathcal{W}_{\text{even}} \otimes \bigwedge \mathcal{W}_{\text{odd}}$ 
as $\Z_2$-graded algebras (with the two operations on the right considered on non-graded vector spaces). 

\begin{rem}
 The graded algebra $\mathbf{S} \mathcal{W}$ can be seen as the set of polynomials in a graded basis of the graded vector space 
$\mathcal{W}$. We write $\overline{\mathbf{S} \mathcal{W}}$ to indicate its completion by allowing formal power series. 
Assume that $(w_\alpha)_\alpha$ is a graded basis of $\mathcal{W}$. Then, any element $s\in \overline{\mathbf{S} \mathcal{W}}$ admits a 
unique decomposition with respect to the $(w_\alpha)_\alpha$ :  
\begin{equation}
\label{EqDecompoSW}
 s=s_0 + \sum_{K=1}^\infty \sum_{\alpha_1\leqslant \ldots \leqslant \alpha_K} 
s_{\alpha_1\ldots \alpha_K} w_{\alpha_1}\ldots w_{\alpha_K},
\end{equation}
where $s_0$ and all the $s_{\alpha_1\ldots \alpha_K}$ are real numbers. We use the multiplicative notation 
$w_{\alpha_1}\ldots w_{\alpha_K}$ to denote $w_{\alpha_1} \otimes \ldots \otimes w_{\alpha_K}$ modulo 
the supercommutativity in $\overline{\mathbf{S}\mathcal{W}}$.
\end{rem}

\subsection{Graded ringed space}
\label{SubSectGradedRingSp}

\begin{defn}
A \emph{$\Z$-graded ringed space} $S$ is a pair $(|S|,\mathcal{O}_S)$ such that  
$|S|$ is a topological space and $\mathcal{O}_S$ is a sheaf of associative unital supercommutative $\Z$-graded rings, called the  
\emph{structure sheaf} of $S$. 
A \emph{$\Z$-graded locally ringed space} is a $\Z$-graded ringed space  
$S=(|S|,\mathcal{O}_S)$ whose stalks $\mathcal{O}_{S,x}$ are local graded rings for all $x \in |S|$.
\end{defn}
\begin{exmp}
 Any locally ringed space is a graded locally ringed space whose sheaf has only elements of degree zero. 
\end{exmp}

\begin{defn}
A \emph{morphism of $\Z$-graded ringed spaces} 
$\phi:(|M|,\mathcal{F})\rightarrow(|N|,\mathcal{G})$ 
is a pair $(|\phi|,\phi^*)$, where $|\phi|:|M| \rightarrow |N|$ is a morphism of topological spaces and 
$\phi^* : \mathcal{G}\rightarrow \phi_* \mathcal{F}$ is a morphism of sheaves of  associative unital supercommutative $\Z$-graded rings, 
which means that it is a collection of morphisms $\phi_V :\mathcal{G}(V)\rightarrow \mathcal{F}(|\phi|^{-1}(V))$ for all $V \in |N|$.

\noindent A \emph{morphism of $\Z$-graded locally ringed spaces} 
$\phi:(|M|,\mathcal{F})\rightarrow(|N|,\mathcal{G})$ is a morphism of $\Z$-graded ringed spaces such that, 
for all $x\in |M|$, the induced morphism on the stalk $\phi_x :\mathcal{G}_{|\phi|(x)}\rightarrow \mathcal{F}_x$ is local, 
which means that 
$\phi_x^{-1}\left( \mathfrak{m}_{_{M,x}} \right) = \mathfrak{m}_{_{N,|\phi|(x)}}$, where  
$\mathfrak{m}_{_{M,x}}$ (respectively $\mathfrak{m}_{_{N,|\phi|(x)}}$) is the maximal homogeneous ideal of $\mathcal{F}_x$ 
(resp. $\mathcal{G}_{|\phi|(x)}$).
\end{defn}

Let $(|M|,\mathcal{F})$ and $(|N|,\mathcal{G})$ be two graded locally ringed spaces. 
Assume that for all $x\in |M|$, there exists an open set  $V\subseteq|M|$ containing $x$ and an open set $\widetilde{V}\subseteq |N|$ 
such that there exists an isomorphism of graded locally ringed spaces
$\phi_V:(V,\restriction{\mathcal{F}}{V})\to(\widetilde{V},\restriction{\mathcal{G}}{\widetilde{V}})$. Then, we say that 
the graded locally ringed space $(|M|,\mathcal{F})$ is \emph{locally isomorphic to} $(|N|,\mathcal{G})$.

\begin{exmp}
 A smooth manifold is a locally ringed space locally isomorphic to $(\R^n,\mathcal{C}^\infty_{\R^{n}})$. This can be rephrased as 
a local isomorphism of graded locally ringed spaces whose  sheaves have only elements of degree zero. 
\end{exmp}

\subsection{Graded domain}
\label{SubSectGradedDomain}

\begin{defn}
\label{DefDomGrad}
Let $(p_j)_{j\in \Z}$ be a sequence of non-negative integers  
 and $\mathcal{W}=\bigoplus_{j} \mathcal{W}_j$ be a $\Z$-graded vector space of dimension $(p'_j)_{j \in \Z}$ with 
 $p'_0=0$ and $p'_j=p_j$ otherwise. 
 We say that $U^{(p_j)}=(U,\mathcal{O}_U)$ is a \emph{$\Z$-graded domain} of dimension $(p_j)_{j\in \Z}$, 
if $U$ is an open subset of $\R^{p_0}$ and for all $V\subseteq U$, 
$\mathcal{O}_U(V) =\mathcal{C}^\infty_{\R^{p_0}}(V)\otimes\overline{\mathbf{S}\mathcal{W}}$.
\end{defn}

\begin{exmp}
Let $(p_j)_{j\in \Z}$ be a sequence of non-negative integers. 
We write $\R^{(p_j)}=\left(\R^{p_0}\right)^{(p_j)}$ to indicate the graded domain of dimension $(p_j)_{j\in \Z}$ 
and topological space $\R^{p_0}$.
\end{exmp}

A \emph{global coordinate system} on $U^{(p_j)}$ is given by a global coordinate system $(t_1,\ldots,t_{p_0})$ on $U$, 
and by a graded basis $(w_\alpha)_\alpha$ of $\mathcal{W}$. We write such a system on a graded domain as $(t_i,w_\alpha)_{i,\alpha}$.
From Equation \eqref{EqDecompoSW}, we see that any section $f\in \mathcal{O}_U(V)$, with $V\subseteq U$, 
admits a unique decomposition in the global coordinate system :  
\begin{equation}
\label{EqDecompo}
 f=f_0 + \sum_{K=1}^\infty \sum_{\alpha_1\leqslant \ldots \leqslant \alpha_K} 
f_{\alpha_1\ldots \alpha_K} w_{\alpha_1}\ldots w_{\alpha_K},
\end{equation}
where $f_0$ and all the $f_{\alpha_1\ldots \alpha_K}$ are smooth functions on $V$.

Let $\mathcal{J}(V)$ be the \emph{ideal generated by all sections with nonzero degree} on the open subset $V$. 
The map $\pi_1:\mathcal{O}_U(V)\to\mathcal{C}^\infty_{U}(V):f\mapsto f_0$, with $f_0$ defined in Equation \eqref{EqDecompo}, admits 
$\mathcal{J}(V)$ as kernel. Moreover, it is a left inverse for the embedding $\mathcal{C}^\infty_{U}(V)\to\mathcal{O}_U(V)$. 

If we write $\pi$ for the canonical projection  $\mathcal{O}_U(V)\to \mathcal{O}_U(V)/\mathcal{J}(V)$, then $\pi_1$ implies the 
existence of an isomorphism between $\mathcal{C}_U^\infty(V)$ and $\mathcal{O}_U(V)/\mathcal{J}(V)$ as shown in the 
following diagram : 
\begin{diagram}
\mathcal{C}^\infty_{U}(V) &&\pile{ \lDotsto^{\pi_1} \\ \rInto}  &&\mathcal{O}_U(V)\\
&\rdTo(3,2)_{\simeq} &&&\dOnto_\pi  \\
&        &&&\mathcal{O}_U(V)/\mathcal{J}(V)
\end{diagram}
The map $\pi_1$ can be used for evaluating a section at a point $x\in V$. It is the \emph{value} at $x$.

\begin{defn}
Let $U^{(p_j)}=(U,\mathcal{O}_U)$ be a $\Z$-graded domain of dimension $(p_j)_{j\in\Z}$ and $V\subseteq U$ be an open subset. 
Let $x\in V$ and $f\in \mathcal{O}_U(V)$. 
The \emph{value} (or \emph{evaluation}) of $f$ at $x$ is defined as $f(x)=\pi_1(f)(x)\in\R$.
\end{defn} 

The \emph{ideal of sections with null value} at $x\in V$, written as $\mathcal{I}_x(V)$,  is  given by 
$\mathcal{I}_x(V)=\left\{ f\in\mathcal{O}_U(V) \big{|}f(x)=0\right\}$. Remark that 
$\mathcal{J}(V)\subseteq\mathcal{I}_x(V)$. 
For any $r\in \N^\times$, we can define $\mathcal{I}_x^r(V)$ as $
\left\{f\in \mathcal{O}_{U}(V)\,\big{|}\,\exists  f_1,\ldots,f_r\in\mathcal{I}_x(V), \, f=f_1\ldots f_r\right\}$.

\begin{lem}\emph{(Hadamard's lemma).}\label{LemHadamardDomGrad} 
 Let $U^{(p_j)}=(U,\mathcal{O}_U)$ be a $\Z$-graded domain of dimension $(p_j)_{j\in\Z}$ with global coordinate system 
$(t_i,w_\alpha)_{i,\alpha}$. Consider an open subset $V\subseteq U$, a point $x\in V$ and a section  $f\in \mathcal{O}_U(V)$. 
Then, for all $k\in \N$, there exists a polynomial $P_{k,x}$ of degree $k$ in the variables 
$\big(t_i-t_i(x)\big)_{i}$ and $(w_\alpha)_{\alpha}$ such that  $f - P_{k,x} \in \mathcal{I}_x^{k+1}(V)$.
\end{lem}
\begin{proof}
Recall that, if $g\in\mathcal{C}_U^\infty(V)$ is a smooth function in $t=(t_1,\ldots,t_{p_0})$, one can take its Taylor series of order 
$r\in\N$ at $x$, written as $T_{x}^r(g;t)$. It is a polynomial of order $r$ that satisfies $g(t)=T_{x}^r(g;t)+S_{x}^r(g;t)$ for 
$S_{x}^r(g;t)$ a sum of elements of the form $h(t) Q_x^r(t)$ with $h$ smooth and $Q_x^r(t)$ a homogeneous polynomial of order $r+1$ 
in the variables $(t_i-t_i(x))_i$.

 From Equation \eqref{EqDecompo}, we can write  $f\in \mathcal{O}_U(V)$ as 
\begin{equation*}
 f=f_0 + \sum_{K=1}^\infty \sum_{\alpha_1\leqslant \ldots \leqslant \alpha_K} 
f_{\alpha_1\ldots \alpha_K} w_{\alpha_1}\ldots w_{\alpha_K}.
\end{equation*}

Since $f_0\in \mathcal{C}^\infty(V)$ and $f_{\alpha_1\ldots \alpha_K}\in \mathcal{C}^\infty(V)$ for all indices, 
we can develop them in Taylor series at $x$ in the variables $t=(t_1,\ldots,t_{p_0})$. 
If, for all $k\in\N$, we define $P_{k,x}$ as  
\begin{equation*}
 P_{k,x}(t)= T^k_{x}(f_0;t) + \sum_{K=1}^k \sum_{\alpha_1\leqslant \ldots \leqslant \alpha_K} 
T^{k-K}_x(f_{\alpha_1\ldots \alpha_K};t) w_{\alpha_1}\ldots w_{\alpha_K},
\end{equation*}
then $P_{k,x}$ is a polynomial of order $k$. In particular, we obtain that 
\begin{eqnarray*}
 f-P_{k,x}(t)&=& S^k_{x}(f_0;t) + \sum_{K=1}^k \sum_{\alpha_1\leqslant \ldots \leqslant \alpha_K} 
S^{k-K}_x(f_{\alpha_1\ldots \alpha_K};t) w_{\alpha_1}\ldots w_{\alpha_K} \\
&&+ \sum_{K= k+1}^\infty \sum_{\alpha_1\leqslant \ldots \leqslant \alpha_K} 
f_{\alpha_1\ldots \alpha_K}(t) \, w_{\alpha_1}\ldots w_{\alpha_K} \,.
\end{eqnarray*}
As $w_\alpha\in \mathcal{I}_x(V)$ for all $\alpha$ and $S_x^r(g;t)\in\mathcal{I}_x^{r+1}(V)$ for any function $g$, we deduce that 
$f - P_{k,x} \in \mathcal{I}_x^{k+1}(V)$. 
\end{proof}

\begin{prop}
\label{CorHadamard}
Let $U^{(p_j)}=(U,\mathcal{O}_U)$ be a $\Z$-graded domain of dimension $(p_j)_{j\in\Z}$ and $V\subseteq U$ be an open subset. Then 
\begin{equation*}
 \bigcap_{k\geqslant1} \bigcap_{x\in V} \mathcal{I}_x^k(V) \, = \, \{0\}\,.
\end{equation*}
\end{prop}
\begin{proof}
We show that, if  $f,g\in\mathcal{O}_U(V)$ are two sections such that for all $k\in\N$, $f-g\in \mathcal{I}_x^{k+1}(V)$ 
for all $x\in V$, then $f=g$.

Consider a global coordinate system $(t_i,w_\alpha)_{i,\alpha}$ on $U^{(p_j)}$. 
Set $h=f-g$ and decompose this section as in Equation \eqref{EqDecompo} :
\begin{equation*}
 h=h_0+ \sum_{K=1}^\infty \sum_{\alpha_1\leqslant \ldots \leqslant \alpha_K} 
h_{\alpha_1\ldots \alpha_K} w_{\alpha_1}\ldots w_{\alpha_K}.
\end{equation*}
Fix $k\in \N$. For all $K\leqslant k$, the condition $h\in I_x^{k+1}(V)$ implies that 
$h_{\alpha_1\ldots \alpha_K} \in I_x^{k+1-K}(V)$. 
Since this is true for all $x\in V$, we have that $h_0=0$ and $h_{\alpha_1\ldots \alpha_K}=0$ on $V$ 
for all indices with $K\leqslant k$.
As $k$ is arbitrary, we conclude that $h=0$ on $V$.
\end{proof}

\begin{prop}
\label{PropCaracTige}
Let $U^{(p_j)}=(U,\mathcal{O}_U)$ be a $\Z$-graded domain of dimension $(p_j)_{j\in\Z}$ with global coordinate system 
$(t_i,w_\alpha)_{i,\alpha}$. Consider an open subset $V\subseteq U$ and a point $x\in V$. 
Then $\mathcal{I}_x(V)=\langle\{t_i-t_i(x),w_\alpha\}_{_{i,\alpha}}\rangle$.
\end{prop}
\begin{proof}
We write $A=\langle \{t_i-t_i(x),w_\alpha\}_{_{i,\alpha}}\rangle$ for the proof. Consider a section $f\in \mathcal{I}_x(V)$. 
Decomposing $f$ as in Equation \eqref{EqDecompo}, we have  
 \begin{equation*}
 f=f_0+ \sum_{K\geqslant 1} \sum_{\alpha_1\leqslant \ldots \leqslant \alpha_K} 
f_{\alpha_1\ldots \alpha_K} w_{\alpha_1}\ldots w_{\alpha_K}.
 \end{equation*} 
From the definition of $A$, we only need to show that $f_0 \in A$. By assumption, $f\in \mathcal{I}_x(V)$, so that its
value at $x$ is zero. Therefore, considering the Taylor series of $f_0$ of order $1$ gives that 
$f_0(t)=\sum_i h_i(t) (t_i-t_i(x))$ for some 
smooth functions $h_i$ with $i\in\{1,\ldots,p_0\}$. Hence $f_0\in A$ since $A$ contains the sections $(t_i-t_i(x))_i$. 
The inclusion $A\subseteq \mathcal{I}_x(V)$ is direct by definition of $A$. 
\end{proof}

\begin{prop}
\label{PropEvaluationValeur}
Consider an open subset $V\subseteq U$ and a point $x\in V$. Then, $\mathcal{O}_{U}(V)=\R\oplus \mathcal{I}_x(V)$. 
Moreover, the projection on constant sections $\ev_x:\mathcal{O}_{U}(V)\to \R$
calculates the value of every section.
\end{prop}
\begin{proof}
Remark that a section $f\in\mathcal{O}_U(V)$ can be written as $f=(f-f(x))+f(x)$, where $f(x)=f(x)\cdot 1_V$ is a multiple of the unit 
section $1_V\in\mathcal{O}_U(V)$. Besides, one can see that the section $f-f(x)$ has value $0$ at $x$. 
Hence $f$ can be written as a sum of a constant section and an element of $\mathcal{I}_x(V)$ and we obtain that  
$\mathcal{O}_{U}(V)\subseteq \R\oplus \mathcal{I}_x(V)$. 
Under these notations, we have $\ev_x(f)=f(x)\cdot 1_V$. The second statement follows directly. 
\end{proof}

At a point $x$ of a graded domain $U^{(p_j)}$, we write the elements of the stalk, called \emph{germs}, as $[f]_x\in\mathcal{O}_{U,x}$. 
The evaluation is defined on germs in the same manner as it is done for sections. 
We denote the ideal generated by all germs with nonzero degree by $\mathcal{J}_x$, 
and the ideal of germs with null value by $\mathcal{I}_x$.  These two ideals can be obtained by inducing 
the ideals $\mathcal{J}(V)$ and $\mathcal{I}_x(V)$ on the stalk.

\begin{rem}
\label{RemPropStalk}
 The statements in Lemma \ref{LemHadamardDomGrad} as well as Propositions \ref{CorHadamard}, \ref{PropCaracTige} and 
\ref{PropEvaluationValeur} can be reformulated for germs. In particular, we get that $\mathcal{I}_x$ is 
a homogeneous ideal such that $\mathcal{O}_{U,x}=\R\oplus\mathcal{I}_x$. 
\end{rem}

\begin{lem}
\label{LemMaximality}
$\mathcal{I}_x$ is the unique maximal homogeneous ideal of $\mathcal{O}_{U,x}$. 
\end{lem}
\begin{proof}
The decomposition $\mathcal{O}_{U,x}=\R\oplus\mathcal{I}_x$ ensures that $\mathcal{I}_x$ is a maximal homogeneous ideal.

Assume that $\mathfrak{m}$ is another maximal homogeneous ideal of $\mathcal{O}_{U,x}$. Then, there exists 
an homogeneous element $[f]_x$ which is in $\mathfrak{m}$ but not in $\mathcal{I}_x$. Since $\mathcal{I}_x$ contains the ideal 
$\mathcal{J}_x$ generated by all germs of nonzero degree, $[f]_x$ is a germ of degree zero. 

By maximality of $\mathcal{I}_x$, the ideal $\mathcal{O}_{U,x} [f]_x + \mathcal{I}_x$ is $\mathcal{O}_{U,x}$. 
Therefore, there exist $[h]_x\in\mathcal{O}_{U,x}$ and $[g]_x\in \mathcal{I}_x$, both of degree zero, such that $1=[h]_x[f]_x+[g]_x$. 
If we take the value of the germs at $x$, the equality gives that $1=[h]_x(x)[f]_x(x)$ since $[g]_x(x)=0$. 
Thus, setting $a=[f]_x(x)$, we have $a\in \R\setminus\{0\}$. According to the decomposition $\mathcal{O}_{U,x}=\R\oplus\mathcal{I}_x$, 
we have that $[f]_x=a+[\tilde{f}]_x$ with some $[\tilde{f}]_x\in\mathcal{I}_x$ of degree zero. 

Take a section $\tilde{f}$ defined in a neighbourhood of $x$ which represents $[\tilde{f}]_x$ and set 
$F=a^{-1}\sum\limits_{k=0}^\infty (-a^{-1}\tilde{f})^k$. On the stalk $\mathcal{O}_{U,x}$, we find that $[F]_x[f]_x=1$. 
Since $\mathfrak{m}$ is an ideal, this equality implies that $1\in\mathfrak{m}$. Thus, we have  $\mathfrak{m}=\mathcal{O}_{U,x}$.
 Hence $\mathcal{I}_x$ is the unique maximal homogeneous ideal of $\mathcal{O}_{U,x}$. 
\end{proof}
\begin{thm}
A $\Z$-graded domain is a $\Z$-graded locally ringed space.
\end{thm}
\begin{proof}
By Definition \ref{DefDomGrad}, a graded domain is a graded ringed space. 
The locality of the stalks follows from Lemma \ref{LemMaximality}. 
\end{proof}

\begin{rem}
\label{RemLocality}
An alternative definition of graded domain can be given. We can set that the sheaf is 
$\mathcal{O}_U(V)=\mathcal{C}^\infty_U(V)\otimes \mathbf{S}\mathcal{W}$, namely we do not consider formal power series in the 
graded coordinates but only polynomials. 
Then, in general, the ideal $\mathcal{I}_x$ of germs with null value is not the unique maximal homogeneous
ideal of $\mathcal{O}_{U,x}$. For example, the homogeneous ideal   $\langle 1-[w]_x \rangle$ 
(where $w$ is a a non-nilpotent element of degree $0$ which is a product of local coordinates of nonzero degree) 
is contained in a maximal homogeneous ideal different from $\mathcal{I}_x$, as the germ $1-[w]_x$ has value $1$. 
Thus, the stalks of this alternative graded domain are not local. Nevertheless, they are $\pi$-local: 
$\mathcal{I}_x$ is the only maximal homogeneous ideal of $\mathcal{O}_{U,x}=\mathcal{C}^\infty_{U,x}\otimes \mathbf{S}\mathcal{W}$ 
which projects onto the maximal ideal of $\mathcal{C}^\infty_{U,x}$. Though this definition can be chosen, it limits the study 
of $\Z$-graded manifolds since a local expression of the form $f(t+w)$ would not always admit a finite power series in $w$, 
which is an obstacle to introduce differentiability.
\end{rem}

\section{Graded manifold}
\label{SectionGradedManifold}

\subsection{Definition} 

\begin{defn}
\label{DefVarieteGraduee}
Let $|M|$ be a Hausdorff secound-countable topological space and  $\mathcal{O}_{M}$ be a sheaf of  associative unital supercommutative 
$\Z$-graded algebras on $|M|$, such that  $M=(|M|,\mathcal{O}_{M})$ is a $\Z$-graded locally ringed space. 
Moreover, let $(p_j)_{j\in \Z}$ be a sequence of non-negative integers.  
We say that $M$ is a \emph{$\Z$-graded manifold} of dimension $(p_j)_{j\in \Z}$, 
if there exists a local isomorphism of $\Z$-graded locally ringed spaces $\phi$ between $M$ and $\R^{(p_j)}$. 
 We say that an open subset $V\subseteq|M|$ is a \emph{trivialising open set} if the restriction of $\phi$ to $V$, written $\phi_V$, 
is an isomorphism of graded locally ringed spaces between $(V,\restriction{\mathcal{O}_{M}}{V})$ and its image 
$\big(\widetilde{V},\restriction{\mathcal{C}^\infty_{\R^{p_0}}}{\widetilde{V}}\big)$, for $\widetilde{V}=\phi_V(V)$.
\end{defn}
If a sequence $(t_i,w_\alpha)_{i,\alpha}$ is a global coordinate system on the $\Z$-graded domain $\R^{(p_j)}$, its 
pullbacks on trivialising open sets form a local coordinate system on $M$.  By abuse of notation, 
one says that $(t_i,w_\alpha)_{i,\alpha}$ is a \emph{local coordinate system} on $M$.
Morphisms of graded manifolds are defined as morphisms of graded locally ringed spaces.

\begin{rem}
\label{RemLocGradMan}
 Alternatively, one can define a graded manifold as a graded $\pi$-locally ringed space 
(\emph{i.e.} stalks are $\pi$-local graded rings), 
locally isomorphic to the alternative graded domains defined in Remark \ref{RemLocality}.  
Full development of such a definition is left to the interested reader.
\end{rem}

\begin{exmp}
\label{ExVarGradFibreQlcq}
Let $M$ be a smooth manifold of dimension $n$ and write $|M|$ for its underlying topological space. 
If $E\to M$ is a vector bundle of rank $m$ and $k$ is a nonzero integer, 
we define the sheaf $\mathcal{O}_{E[k]}$ on $|M|$ by  
\begin{itemize}
 \item[--] $\mathcal{O}_{E[k]}(V)$ equals $\Gamma(V,\restriction{\bigwedge E^*}{V})$ if $k$ is odd, 
 \item[--] $\mathcal{O}_{E[k]}(V)$ equals $\Gamma(V,\restriction{\overline{\mathbf{S} E^*}}{V})$ if $k$ is even,
\end{itemize}
for all open subsets $V\subseteq |M|$. Here, 
the algebra bundles $\bigwedge E^*$ and $\overline{\mathbf{S}E^*}$ are considered in the usual non-graded sense.
 Then $\mathfrak{G}_{E[k]}=(|M|,\mathcal{O}_{E[k]})$ is a graded manifold of dimension $(p_j)_{j\in \Z}$, 
with $p_0=n$, $p_k=m$ and $p_j=0$ if $j\in \Z\setminus\{0,k\}$. The grading on the sections is defined by $|v|=kl$ if  
 $v\in\Gamma\left(M,\mathbf{S}^l E^*\right)$ or $v\in \Gamma\big(M,\bigwedge ^l E^*\big)$.
\end{exmp}

The graded manifold in Example \ref{ExVarGradFibreQlcq} is usually denoted by $E[k]$. We prefer to use a different notation in 
these notes so that such a graded manifold obtained from a shift is written differently than a lifted graded vector space. 

\begin{rem}
One can define $\N$-graded manifolds as $\Z$-graded manifolds whose dimension is indexed by $\N$. Therefore, 
any homogeneous section has nonnegative degree. This leads to an interesting property: on a $\N$-graded manifold, there does not exist 
a section of degree zero which can be obtained as a product of sections of nonzero degree.  This means that the locality and 
the $\pi$-locality of the stalks are equivalent conditions. Hence, $\N$-graded manifolds do not require the introduction of 
formal series to be studied, unlike $\Z$-graded manifolds (see Remarks \ref{RemLocality} and \ref{RemLocGradMan}). 
Another difference between $\N$- and $\Z$-graded manifolds is that there exists a Batchelor-type theorem on 
$\N$-graded manifolds \cite{BonavolontaPoncin}. It is not known if an analogous result holds for $\Z$-graded manifolds.
\end{rem}

\subsection{Properties} 
\label{SubSectProperties}

Let $M=(|M|,\mathcal{O}_{M})$ be a $\Z$-graded manifold and take a point $x\in |M|$. By definition, the stalk $\mathcal{O}_{M,x}$ is a 
local graded ring which admits a maximal homogeneous ideal $\mathfrak{m}_x$. Using the isomorphism $\phi$ of a trivialising open 
set around $x$ with a graded domain $(U,\mathcal{O}_U)$, we get two canonical algebra isomorphisms :  
$\mathcal{O}_{M,x}/\mathfrak{m}_{x}\simeq \mathcal{O}_{U,\phi(x)}/\mathcal{I}_{\phi(x)}\simeq \R$. We set 
$\ev_x:\mathcal{O}_{M,x}\to \R$ to denote the composition of the projection $\mathcal{O}_{M,x}\to\mathcal{O}_{M,x}/\mathfrak{m}_{x}$ 
with the above isomorphism.

\begin{defn}
\label{DefValeur}
Let $M=(|M|,\mathcal{O}_M)$ be a $\Z$-graded manifold of dimension $(p_j)_{j\in\Z}$, 
$V\subseteq |M|$ an open subset and $x\in V$. Set $\mathfrak{m}_x$ to indicate the maximal ideal of $\mathcal{O}_{M,x}$.
For every section $f\in \mathcal{O}_M(V)$, the \emph{value} (or \emph{evaluation}) of $f$ at $x$ is given by $f(x)=\ev_x([f]_x)$.
\end{defn}

From the definition above, the ideal of sections with null value $\mathcal{I}_x$ at $x$ can be defined on the graded manifold as 
the kernel of $\ev_x$. Moreover, if $V\subseteq |M|$ is a trivialising open set,
Hadamard's Lemma \ref{LemHadamardDomGrad} and Proposition \ref{CorHadamard} hold on $V$.

Although it is not used for defining the value on a graded manifold, the ideal generated by all sections with nonzero degree exists. 
We write $\mathcal{J}(V)$ to denote this ideal on an open subset $V$ of a graded manifold $M$. If $V$ is a trivialising open set, 
for any section $f\in\mathcal{O}_{M}(V)$ we set $\tilde{f}$ to indicate the element $\pi_1(f)\in \mathcal{C}^\infty(V)$ 
(see \ref{SubSectGradedDomain}).
Under this notation, the following partition of unity holds true :
\begin{prop}
\label{PropPartUnite}
Let $M=(|M|,\mathcal{O}_M)$ be a $\Z$-graded manifold of dimension $(p_j)_{j\in\Z}$ and $(V_\beta)_{\beta}$ be an open covering of $|M|$ by 
trivialising open subsets. 
Then there exists a sequence $(g_\beta)_{\beta}\subset \mathcal{O}_M(|M|)$ such that for all $\beta$,
\begin{enumerate}
 \item $g_\beta \in \left(\mathcal{O}_M(|M|)\right)_0$ ,
 \item $\text{\emph{supp}}\left( g_\beta \right)\subset V_\beta$ ,
 \item $\sum g_\beta=1$ and $\tilde{g}_\beta \geqslant 0$.
\end{enumerate}
\end{prop}
The proof of this proposition is exactly the same that the one for the partition of unity on a supermanifold 
\cite{CCFsupersym}. 

Let $M=(|M|,\mathcal{O}_M)$ be a graded manifold. 
We can use Proposition \ref{PropPartUnite} to show that the presheaf  
$\mathcal{O}_M/\mathcal{J}:V\mapsto \mathcal{O}_M(V)/\mathcal{J}(V)$, where $V\subseteq |M|$, is a sheaf. 
Note that $\mathcal{O}_M$ is $\pi$-local since it is local. Therefore the stalks of $\mathcal{O}_M/\mathcal{J}$ are local. 
Hence, $\mathcal{O}_M/\mathcal{J}$ is a sheaf locally isomorphic to  $\mathcal{C}^\infty_{\R^{p_0}}$ and we 
get that $(|M|,\mathcal{O}_M/\mathcal{J})$ is a smooth manifold. Thus, one can assume from the beginning that a graded manifold has an 
underlying structure of smooth manifold. 
This justifies the frequent definition of graded manifold found in the literature, which considers  
graded manifolds as smooth manifolds with an additional \emph{glued} structure.

\begin{rem} \label{RemZeroDeg}
 In recent developments (see \cite{Vogl16,MehtaStiXu}) formal polynomial functions on graded manifolds have been considered with 
respect to a $\Z$-graded vector space that has a non-trivial part of degree zero. 
This means that, in Definition \ref{DefVarieteGraduee}, we could choose $\mathcal{W}$ such that  $\text{dim}(\mathcal{W}_0)\geqslant 1$. 
 In that case, the underlying structure of 
smooth manifold still exists, but one has to obtain it from the morphism $\pi_1$ instead of $\pi$ (see \ref{SubSectGradedDomain}) 
as their images are no longer isomorphic. 
\end{rem}

\subsection{Vector fields}

Let $M=(|M|,\mathcal{O}_M)$ be a $\Z$-graded manifold. 
We say that a \emph{derivation of degree $k$} on an open subset $V\subseteq|M|$ is an application 
$X_V^k:\mathcal{O}_M(V)\to \mathcal{O}_M(V)$ which maps every section of degree $m$ onto a section of degree $m+k$ and which 
satisfies 
\begin{equation}
\label{EqDerivation}
X_V^k(fg)=X_V^k(f)\,g+(-1)^{k |f|}f\,X_V^k(g), \text{ }\forall f,g\in \mathcal{O}_M(V). 
\end{equation} 

\begin{defn}
A \emph{vector field} $X$ on a $\Z$-graded manifold $M=(|M|,\mathcal{O}_M)$ of dimension $(p_j)_{j\in\Z}$ 
 is a $\R$-linear derivation of $\mathcal{O}_M$, 
\emph{i.e.} a family of mappings $X_V:\mathcal{O}_M(V)\to \mathcal{O}_M(V)$ for all open subsets $V\subseteq|M|$, 
compatible with the sheaf 
restriction morphism  and that satisfies $X_V=\sum_{k\in \Z} X_V^k$, where $X_V^k$ is a derivation of degree $k$. 
We say that $X$ is \emph{graded} of degree $k$ if $X_V=X_V^k$ for all $V\subseteq |M|$.

\noindent The \emph{graded tangent bundle} $\Vect_M$ of $M$ is the sheaf whose sections are the vector fields. 
It assigns to every open subset  $V\subseteq|M|$ the graded vector space of derivations on $\restriction{\mathcal{O}_M}{V}$.
\end{defn}

First, we study the graded tangent bundle on a graded domain.  
Assume that $U^{(p_j)}=(U,\mathcal{O}_U)$ is a graded domain of dimension $(p_j)_{j\in\Z}$ with global coordinate system  
$(t_i,w_\alpha)_{i,\alpha}$. We have that $\mathcal{O}_U=\mathcal{C}^\infty_U\otimes\overline{\mathbf{S}\mathcal{W}}$ 
for some graded vector space $\mathcal{W}$ with graded basis $(w_\alpha)_{\alpha}$. 
Consider the basis $\left(\frac{\partial}{\partial t_i}\right)_i$ of the non-graded tangent bundle $\Vect_{U}$ 
and set $\left(\frac{\partial}{\partial w_\alpha}\right)_\alpha$ for the dual graded basis of $(w_\alpha)_\alpha$ in $\mathcal{W}^*$. 
These objects satisfy, for all indices $i,j,\alpha,\beta$, 
\begin{eqnarray}
\quad\frac{\partial}{\partial t_i}\, (t_j)=\delta_{ij}\,,\quad \quad\frac{\partial}{\partial t_i}\, (w_\beta)=0 \,, \label{EqCDVti}\\
\quad\frac{\partial}{\partial w_\alpha}\, (w_\beta)=\delta_{\alpha\beta}\,,\quad \quad\frac{\partial}{\partial w_\alpha}\,(t_j)=0 
\label{EqCDVwalpha}\,.
\end{eqnarray}
Moreover, we can extend them as $\R$-linear derivations on the whole sheaf by Leibniz's rule \eqref{EqDerivation}. 
These are graded vector fields and the degree of $\frac{\partial}{\partial t_i}$ is $0$ while the degree of 
$\frac{\partial}{\partial w_\alpha}$ is $-|w_\alpha|$. 

\begin{lem}
\label{LemDecompoCDV}
Let $U^{(p_j)}=(U,\mathcal{O}_U)$ be a $\Z$-graded domain of dimension $(p_j)_{j\in\Z}$. 
Then its graded tangent bundle $\Vect_{U^{(p_{j})}}$ is a free sheaf of 
$\mathcal{O}_U$-modules such that 
\begin{equation}
\label{EqBaseCDV}
 \Vect_{U^{(p_{j})}} \, \simeq \, \bigg(\mathcal{O}_U \underset{\mathcal{C}^\infty_U}{\otimes} 
\Vect_{U}\bigg) \,\, \oplus \,\, \left(\mathcal{O}_U \underset{\R}{\otimes} \mathcal{W}^*\right) \,.
\end{equation}
If $(t_i,w_\alpha)_{i,\alpha}$ is a global coordinate system on $U^{(p_j)}$, the vector field $X$ admits the following splitting :
\begin{equation*}
X=\sum_{i}\, X_{0,i}\frac{\partial}{\partial t_i}
+\sum_{\alpha}\, X_{1,\alpha}\frac{\partial}{\partial w_\alpha}\,,
\end{equation*}
where $X_{0,i}=X(t_i)\in\mathcal{O}_U(U)$ and $X_{1,\alpha}=X(w_\alpha)\in \mathcal{O}_U(U)$ for all $i,\alpha$.
\end{lem}
\begin{proof}
 Let $X\in \Vect_{U^{(p_{j})}}$ be a vector field. 
For all open subsets $V\subseteq U$ and every $f\in\mathcal{O}_U(V)$, 
 $f X$ is a derivation on $\mathcal{O}_U(V)$.
Therefore $\Vect_{U^{(p_{j})}}$ is a sheaf of $\mathcal{O}_U$-modules. The fact that it is free and 
Equation \eqref{EqBaseCDV} are consequences of the existence of a graded basis, which is obtained below.

Assume that the global coordinates on $U^{(p_j)}$ are $(t_i,w_\alpha)_{i,\alpha}$ and 
let $X\in \Vect_{U^{(p_{j})}}$ be a vector field. We define 
\begin{equation*}
 \hat{X}=\sum \limits_{i} X(t_i) \frac{\partial}{\partial t_i} + \sum \limits_{\alpha} 
X(w_\alpha) \frac{\partial}{\partial w_\alpha},
\end{equation*}
where $X(z)$ is the section in $\mathcal{O}_{U}(U)$ given by applying $X$ to the section $z$. Since $\hat{X}$ is a derivation, 
$D=X-\hat{X}$ is also one and it vanishes on the global coordinates $(t_i,w_\alpha)_{i,\alpha}$. By Leibniz's rule \eqref{EqDerivation}, 
$D$ is equal to zero on all polynomials generated by the global coordinates.  

Our aim is to show that $D$ is null on all sections. Take an open subset $V\subseteq U$ and a section $f\in\mathcal{O}_U(V)$. From 
Lemma \ref{LemHadamardDomGrad} we have that, for all $n\geqslant 0$, there exists a polynomial 
$P_{n,x}$ of degree $n$ in the global coordinates such that $f-P_{n,x}\in \mathcal{I}_x^{n+1}(V)$. We set $h=f-P_{n,x}$.  

Notice that for any vector field $Y$ we have $Y(\mathcal{I}_x^n(V))\subset\mathcal{I}_x^{n-1}(V)$ for  $n\geqslant1$.
Since $D$ vanishes on every polynomial in $(t_i)_i$ and $(w_\alpha)_\alpha$, we have by linearity that $D(f)=D(h)\in \mathcal{I}_x^n(V)$. 
Since this is true for all $n\geqslant 1$ and any point $x\in V$,  Proposition 
\ref{CorHadamard} gives that $D(f)=0$ on $U$. As $f$ is any section of $\mathcal{O}_U(V)$ and $V$ is any open set, 
we have proved that $D=0$.

It remains to show that the decomposition is unique. Assume that $\tilde{X}$ is another decomposition of $X$ such that 
$\tilde{X}=\sum\limits_{i}f_i\frac{\partial}{\partial t_i}+\sum\limits_{\alpha}g_\alpha\frac{\partial}{\partial w_\alpha}.$ 
The relations $(\tilde{X}-\hat{X})(t_i)=0$ and $(\tilde{X}-\hat{X})(w_\alpha)=0$ implies that $f_i=X(t_i)$ and 
$g_\alpha=X(w_\alpha)$.
\end{proof}

The following result is a consequence of Definition \ref{DefVarieteGraduee} and Lemma \ref{LemDecompoCDV} :
\begin{thm}
\label{CorCaracBaseCDV}
Let $M=(|M|,\mathcal{O}_M)$ be a $\Z$-graded manifold of dimension $(p_j)_{j\in\Z}$. 
Then ${\Vect_M}$  is a locally free sheaf of $\mathcal{O}_M$-modules such that, 
for any trivialising open subset $V\subseteq |M|$, its restriction to $V$ satisfies 
\begin{equation*}
 \restriction{\Vect_{M}}{V} \, \simeq \, \bigg(\restriction{\mathcal{O}_M}{V} \underset{\mathcal{C}^\infty_V}{\otimes} 
\Vect_{V}\bigg) \,\, \oplus \,\, \left(\restriction{\mathcal{O}_M}{V} \underset{\R}{\otimes} \mathcal{W}^*\right) \,.
\end{equation*} 
If $(t_i,w_\alpha)_{i,\alpha}$ is a local coordinate system on $M$, the vector field $X$ admits the following splitting 
on a trivialising open set $V$ :
\begin{equation*}
X=\sum_{i}\, X_{0,i}\frac{\partial}{\partial t_i}
+\sum_{\alpha}\, X_{1,\alpha}\frac{\partial}{\partial w_\alpha}\,,
\end{equation*}
where $X_{0,i}=X(t_i)\in \mathcal{O}_M(V)$ and $X_{1,\alpha}=X(w_\alpha)\in \mathcal{O}_M(V)$ for all $i,\alpha$.
\end{thm}
\begin{rem}
 The graded tangent bundle of a graded manifold can be turned into a sheaf of graded Lie algebras if it is endowed 
with the graded commutator $[a,b]=ab-(-1)^{|a||b|}ba$. 
\end{rem}

Given a graded manifold $M=(|M|,\mathcal{O}_M)$, 
the map $\ev_x:\mathcal{O}_{M,x}\to \R$ allows us to define tangent vectors at a point $x\in M$. 
To see it, first remark that any vector field $X$ can be induced on the stalk. Then, if $X$ is defined locally by 
$X=\sum_{i}\, f_i\frac{\partial}{\partial t_i}+\sum_{\alpha}\, g_\alpha \frac{\partial}{\partial w_\alpha}$, 
its restriction to the stalk $[X]_x:\mathcal{O}_{M,x}\to \mathcal{O}_{M,x}$ is given by   
\begin{eqnarray*}
 [X]_x= \sum_i [f_i]_x \,\, \left[\frac{\partial}{\partial t_i}\right]_{x}
+\sum_{\alpha}\, [g_\alpha]_x \,\, \left[\frac{\partial}{\partial w_\alpha}\right]_{x},
\end{eqnarray*}
where $\left[\frac{\partial}{\partial t_i}\right]_{x}$ and $\left[\frac{\partial}{\partial w_\alpha}\right]_{x}$ satisfy 
Equations \eqref{EqCDVti} and \eqref{EqCDVwalpha} induced on $\mathcal{O}_{M,x}$, respectively. 
The \emph{tangent vector associated to} $X$, written $X_x$, is the application $X_x=\ev_x\circ [X]_x:\mathcal{O}_{M,x}\to \R$. 
More generally, tangent vectors are defined from derivations of the stalk by composition with $\ev_x$. 
The graded vector space of tangent vectors at $x$ is the \emph{graded tangent space}, denoted by $T_xM$.

\begin{rem}
It is usually not possible to recover a vector field from the tangent vectors that it defines at all points. Indeed, assume 
that a graded manifold $M$ admits $(t_1,w_1)$ as local coordinates of respective degrees $0$ and $1$. 
Consider the vector field $X$ given locally by 
$X=\frac{\partial}{\partial t_1}+w_1 \frac{\partial}{\partial w_1}$. Then the 
induced tangent vector at $x\in M$ satisfies $X_x=\restriction{\frac{\partial}{\partial t_1}}{x}$, 
where $\restriction{\frac{\partial}{\partial t_1}}{x}=\ev_x\circ \left[\frac{\partial}{\partial t_1}\right]_{x}$.
\end{rem}

\section{Application to differential graded manifolds}
\label{SectionApplication}

\begin{defn}
Let $M=(|M|,\mathcal{O}_M)$ be a $\Z$-graded manifold of dimension $(p_j)_{j\in\Z}$. 
A vector field $Q \in \Vect_M$ is 
\emph{homological} if it has degree $+1$ and commutes with itself under the graded commutator : $[Q,Q]=2Q^2=0$. 

\noindent A $\Z$-graded manifold $M$ endowed with a homological vector field $Q$ is said to be a 
\emph{differential graded manifold} or a \emph{$Q$-manifold}. We write it $(M,Q)$ and we refer to it as a \emph{dg-manifold}. 
\end{defn}
\begin{exmp}
Let $M$ be a smooth manifold of dimension $n$ and let $\Omega$ be the sheaf of its differential forms.  
The graded manifold $\mathfrak{G}_{TM[1]}=(|M|,\Omega)$ has dimension $(p_j)_{j\in\Z}$, 
with $p_j=n$ if $j=0,1$ and $p_j=0$ otherwise.
If $(x_i)_{i=1}^n$ are local coordinates on $M$, then $(x_i,dx_i)_{i=1}^n$ is a local coordinate system on 
$\mathfrak{G}_{TM[1]}$ whose degrees are respectively $0$ and $1$. 
The exterior differential $d$ is a vector field on $\mathfrak{G}_{TM[1]}$, 
since its restriction on any open set $V\subseteq |M|$ is a derivation of $\Omega(V)$. 
It is given locally by $d=\sum_i dx_i \frac{\partial}{\partial x_i}$. One can compute that $|d|=1$ and $[d,d]=2\, d^2=0$.
Therefore, $(\mathfrak{G}_{TM[1]},d)$ is a dg-manifold.
\end{exmp}

Several geometrico-algebraic structures can be encoded in terms of dg-manifolds. The easiest example is certainly the Lie algebra 
structures on a vector space.

\begin{thm}
\label{ThmLieAlg}
Let $\mathcal{V}_0$ be a vector space. Then the structures of Lie algebra on $\mathcal{V}_0$ are in one-to-one correspondence 
with the homological vector fields on 
$\mathfrak{G}_{\mathcal{V}_0[1]}=\left(\{\ast\},\mathbf{S} \left(\mathcal{V}_0^*[1]\right) \right)$.
\end{thm}
If $\mathcal{V}_0$ is a Lie algebra, one can show that its Chevalley-Eilenberg differential is a vector field of degree $1$ 
which squares to zero. It endows $\mathfrak{G}_{\mathcal{V}_0[1]}$ with a homological vector field. 
Conversely, one can construct a Lie bracket, known as the derived bracket, from the homological vector field. 
This construction is detailed in \cite{KosmannYvette}.

In fact, Theorem \ref{ThmLieAlg} can be seen as a corollary of two more general correspondences. If we consider an arbitrary 
$\Z$-graded vector space, we have  
\begin{thm}
\label{ThmSHLieAlg}
Let $\mathcal{V}$ be a $\Z$-graded vector space such that $\text{dim}(\mathcal{V}_i)=0$ for all $i\leqslant 0$. 
Then the structures of strongly homotopy Lie algebra on $\mathcal{V}$ are in one-to-one correspondence 
with the homological vector fields on  
$\mathfrak{G}_{\mathcal{V}}=\big(\{\ast\},\overline{\mathbf{S} \left(\Pi \mathcal{V}^*\right)} \big)$. 
\end{thm}
Our definition of $\mathcal{V}$ being a strongly homotopy Lie algebra is that $\mathcal{V}[-1]$ admits a sh-Lie structure 
as stated by Lada and Stasheff in \cite{LadaStasheff}. The proof for Theorem \ref{ThmLieAlg} strictly follows 
the correspondence of such structures with degree $+1$ differentials on the free graded algebra 
$\mathbf{S} \left(\Pi \mathcal{V}^*\right)$. It was proved in \cite{LadaStasheff} that each strongly homotopy Lie algebra 
implies the existence of a differential, while the converse can be found in \cite{LadaMarkl}. 
We refer to \cite[\S 6.1]{SatiSchreiStash} for additional details on this relation. 
Note that Voronov has provided an analogous definition in $\Z_2$-settings, which gives an equivalent result for supermanifolds 
 \cite{VoronovDerivedBrackets,VoronovHigherBrackets}.

An alternative generalization of Theorem \ref{ThmLieAlg} is to allow a non-trivial underlying topological space :

\begin{thm}
\label{ThmLieAlgebroid}
Let  $E\to M$ be a real vector bundle. Then the structures of real Lie algebroid on $E$ are in one-to-one correspondence 
with the homological vector fields on $\mathfrak{G}_{E[1]}=\big(|M|,\Gamma(|M|,\bigwedge E^*)\big)$.
\end{thm}
 This theorem is due to Va\u{\i}ntrob \cite{Vaintrob} in the framework of supermanifolds. 
The proof is exactly the same in the $\Z$-graded case, since there are no local coordinates of even nonzero degree 
on $\mathfrak{G}_{E[1]}$ (see \emph{e.g.} \cite{KiselevVanDeLeur}). 

\vspace{0.2cm}

To conclude, the interested reader is invited to find in \cite{BonavolontaPoncin} another general equivalence. 
It links Lie $n$-algebroids to $\N$-graded manifolds (with a homological vector field and generators of degree at most $n$), 
for all $n\in\N^\times$.

\vspace{0.2cm}

\textbf{Acknowledgments.} 
This work is based on the master's thesis submitted by 
the author at the Université Catholique de Louvain in June 2015. 
He is very grateful to Jean-Philippe Michel for lots of useful discussions and for his remarks on the different versions 
of this paper, and to Yannick Voglaire for bringing the material covered in Remark \ref{RemZeroDeg} to his attention.  
The author also thanks the referees for useful suggestions. 
Part of the writing up of this paper was done with the support of a \textit{University of Leeds 110 Anniversary Research Scholarship}.


\bibliographystyle{plain}
 \bibliography{ITGG}


\end{document}